 \documentclass[11pt]{article}
\usepackage{amsmath}
\usepackage{amssymb} 
\usepackage{latexsym} 
\usepackage{theorem}

\usepackage{txfonts}
\usepackage{euscript}
\usepackage[active]{srcltx}

\newcommand{\msf}[1]{\mathsf {#1}}

\newcommand{\mcal}[1]{{\mathcal {#1}}} 

\newtheorem{theorem}{Theorem}  [section]
      
\newtheorem{lemma}[theorem]{Lemma}
\newtheorem{corollary}[theorem]{Corollary}

{\theoremstyle{break}	
{\theorembodyfont{\rmfamily} 
				\newtheorem{example}[theorem]{Example}
}}

\newenvironment{proof}[1]{\smallskip \noindent {\bf #1}}{\qed\smallskip}

\def\qed{\ifhmode\unskip\nobreak\fi\ifmmode\ifinner\else\hskip5pt\fi\fi
 \hfill\hbox{\hskip5pt\vrule width4pt height6pt depth1.5pt\hskip1pt}}
 
\newcommand{\RR}{\ensuremath{{\mathbb R}}}     
\newcommand{\R}[1]{\ensuremath{{\mathbb R}^{#1}}} 

\newcommand{\FF}{\ensuremath{{\mathbb F}}}     
\newcommand{\HH}{\ensuremath{{\mathbb H}}}    

\newcommand{\CC}{\ensuremath{{\mathbb C}}}     

\newcommand{\ZZ}{\ensuremath{{\mathbb Z}}}	   
\newcommand{\NN}{\ensuremath{{\mathbb N}}}	   
\newcommand{\Np}{\ensuremath{{\mathbb N}_{+}}}  
\renewcommand{\S}[1]{\ensuremath{{\bf S}^{#1}}} 

\newcommand{\D}[1]{\ensuremath{{\bf D}^{#1}}} 

\newcommand{\pde}[2]
	{\ensuremath{\frac{\partial #1}{\partial #2}}}	

\def\d#1dt{\frac{d#1}{dt}}    

\newcommand{\lam}{\ensuremath{\lambda}}

\newcommand{\gam}{\ensuremath{\gamma}}
\newcommand{\Fix}[1]{\ensuremath{\operatorname{\mathsf {Fix}}(#1)}}
\newcommand{\co}{\colon\thinspace} 

\renewcommand{\gg}{\ensuremath{\mathfrak g}}

\newcommand{\dd}{\ensuremath{\mathfrak d}} 

\newcommand{\kk}{\ensuremath{\mathfrak k}}

\renewcommand{\ll}{\ensuremath{\mathfrak l}} 

\renewcommand{\ss}{\ensuremath{\mathfrak {s}}}
 
\newcommand{\ttt}{\ensuremath{\mathfrak t}}

\newcommand{\vv}{\ensuremath{\mathfrak v}}

\newcommand {\w}[1] {\ensuremath{\widetilde {#1}}}
\newcommand{\p}{\ensuremath{\partial}}

\renewcommand{\ker}{\operatorname{\msf {ker}}}

\begin{document}

 \pagestyle{plain}
\date{~}
 \author{\Large  Morris W. Hirsch\\
 Mathematics Department\\
University of Wisconsin at Madison\\
University of California at Berkeley}
 \title{Smooth actions of Lie groups and Lie algebras on manifolds\\
}
 \maketitle
\begin{flushleft}
\small \sf
Journal of  Fixed Point Theory and Applications\\
Volume 10 (2011), pp. 219--232\\
Published online November 12, 2011
\end{flushleft}

\

\begin{center} {\em Dedicated to Professor Richard Palais with warmest
    regards \\ on the
    occasion of his eightieth birthday.} 
\end{center}

\

\begin{abstract}Necessary or sufficient conditions are presented for
  the existence of various types of actions of Lie groups and Lie
  algebras on manifolds.
  \end{abstract}




\tableofcontents

\section{Introduction} \label{sec:intro}

Lie algebras were introduced by Sophus Lie under the name
``infinitesimal group,'' meaning  a finite dimensional
transitive Lie algebra of analytic vector fields on an open set in $\R
n$.
In his 1880 paper {\em Theorie der Transformationsgruppen}
  \cite{Lie80,  Hermann65} and the later book with F. Engel
  \cite{LieEngel93}, Lie classified infinitesimal groups acting in dimensions
  one and two up to analytic coordinate changes.  This  work
  stimulated much research, but attention soon shifted to the
  structure, classification and representation of abstract Lie algebras
  and Lie groups.   

There are
relatively few papers on 
nonlinear actions by noncompact Lie groups (other than $\RR$ and $\CC$). A
selection is included in the References.

We avoid the important but difficult classification problems, looking
instead for connections between algebraic invariants of Lie algebras,
topological invariants of manifolds, and dynamical properties of
actions.  The motivating questions are whether a given Lie group or Lie
algebra acts effectively on a given manifold, how smooth such actions
can be,  and what can be said about orbits and kernels.

\subsection{Background}
In 1950  Mostow \cite{Mostow50} completed  Lie's
program of classifying effective transitive surface actions.
One of his major results is:
\begin{theorem}	[{\sc Mostow}]  \label{th:mostow}
A surface $M$ without boundary admits a transitive Lie group
 action iff\footnote
{We use the late Professor Halmos'  symbol ``iff''  for
   ``if and only if.''}
$M$ is a plane, sphere, cylinder,  torus,
projective plane, M\"obius strip or 
 Klein bottle.\footnote
 {For each equivalence class of transitive surface actions, Mostow
 gives a representative basis of
 vector fields.  Determining which of these classes contains a
specified Lie algebra can be nontrivial.  Here the
 succinct summary  in M. Belliart
 \cite{Belliart97} is helpful.}
\end{theorem}
\noindent By a curious coincidence these are the only surfaces
without boundary admitting effective actions of $SO(2)$, according to
a well known folk
theorem.\footnote
{The key points in the proof are that the action is isometric
for any metric obtained by averaging a Riemannian metric, and the
existence of arcs transverse to nonconstant orbits (Whitney
\cite{Wh33}).}

We mention 
a far-reaching extension of Theorem \ref{th:mostow} that deserves
to be better known:   

\begin{theorem}         \label{th:vinberg}
Let $G$ be a Lie group and $H$ a closed subgroup such that
  $G/H$ is compact.  Then $\chi(G/H)\ge 0$, and if $\chi(G/H) > 0$
  then the fundamental group of $G/H$ is finite.
\end{theorem}
This is due to  
 Gorbatsevich  {\em et al.\ }\cite{Vinberg77} 
(Part II,  Chap.{\ }5,  p{.\ }174,
 {Cor.\ }1). 
 See also Hermann \cite{Hermann65}, 
 Felix {\em et al.\ }\cite[Prop. 32.10]{Halperin01},
Halperin \cite{Halperin77},
Mostow \cite{Mostow05}, Samelson \cite{Samelson46}.

\paragraph{Acknowledgments}  Thanks to the following people for 
valuable advice: A. Adem, M. Belliart, M. Brown, J. Hubbard, R. Kirby,
D. B. A. Epstein, M. Belliart, W. Goldman, R. Gompf, J. Harrison,
K. Kimpe, G. Mostow, F. Raymond, J. Robbin, D. Stowe, W. P. Thurston,
F.-J. Turiel, A. Weinstein.

\subsection{Terminology} \label{sec:term}
  The sets of integers,
positive integers and natural numbers are denoted by $\ZZ$,
$\Np=\{1,2,\dots\}$ and $\NN=\Np\cup 0$ respectively.  $i, j, k, l, m,
n,r$ denote natural numbers, assumed positive unless the contrary is
indicated. 
  $\FF$ stands for the real field $\RR$ or the complex field $\CC$.
Vector spaces and Lie algebras are real and
finite dimensional, and manifolds and Lie groups are connected, unless
otherwise noted.    
The kernel of a homomorphism $h$ is denoted by $\ker (h)$.

A topological manifold is a locally Euclidean metric space.  Unless
otherwise noted, manifolds are assumed to be analytic.
$M^n$ denotes a real or complex analytic manifold having dimension $n$
over the ground field $\RR$ or $\CC$.  The boundary of $M$ is $\p M$.
Except as otherwise indicated, manifolds  are connected
and maps between manifolds are $C^\infty$.  The tangent vector space
to $M$ at $p\in M$ is $T_p(M)$.  
A vector field on $M$ is always assumed to be tangent to $\p M$. 

``Group'' and ``algebra'' are shorthand for ``Lie group'' and ``Lie
algebra''.  $G$ denotes a Lie group with Lie algebra $\gg$ and
universal covering group $\w G$.  Groups are assumed connected unless
the contrary is indicated.  The subscript ``$\circ$'' denotes the
identity component.  Lie groups are named by capital Roman letters and
their Lie algebras are named by the corresponding lowercase gothic
letters.

$GL(m, \FF)$ is the group of $m\times m$ invertible matrices over
$\FF$; its Lie algebra is $\gg\ll (m,\FF)$.  The subgroup of
unimodular matrices is $SL (n,\FF)$, and that of the unimodular upper
triangular matrices is $ST (m,\FF)$. The $k$-fold direct product
$G\times\dots\times G$ is denoted by $G^k$.
The universal covering group of $G$ is $\w G$.

The commutator subgroup of $G$ is $G'$.  The upper central series is 
recursively defined by  $G^{(0)}=G$, \ $G^{(j+1)}=(G^{(j)})'$, with  
corresponding Lie algebras $\gg^{(j)}$.  Recall
 that $G$  and  $\gg$ are   {\em solvable}, of {\em derived length} $l=\ell
 (\gg)=\ell (G)$, if $l\in\Np$ is the smallest number satisfying
 $\gg^{(l)}=0$.  For example, $\ell (\ss\ttt (m,\FF)) = m+1$.

$G$ and $\gg$ are   {\em nilpotent} if there exists $k\in\NN$ such that
$\gg_{(k)}=\{0\}$, where $\gg_{(0)}=\gg$ and
$\gg_{(j+1)}=[\gg, \gg_{(j)}]$.   It is  known that  $\gg$ is solvable
if and only $\gg'$ is nilpotent (Jacobson \cite[Corollary
  II.7.2]{Jacobson62}).  

\paragraph{Actions and local actions}
An {\em action} $\alpha$ of $G$ on  $M$, 
denoted by $(G,M,\alpha)$,
is a homomorphism $g\mapsto g^\alpha$ from $G$ to the group of
homeomorphisms of $M$, having  a continuous {\em evaluation map}
\[
  \mathsf{ev}_\alpha \co G\times M\to M, \,(g,x)\mapsto g^\alpha(x). 
\] 
The action is called  $C^s$ when
$\mathsf{ev}_\alpha$ is differentiable of class $C^s$, where $s\in
\NN$, $s=\infty$, or $s= \omega$ (meaning analytic).
``Smooth'' is a synonym for $C^\infty$.  A {\em flow} is an action of
$\RR$. 

If $1\le r\le \omega$, a $C^r$ {\em Lie algebra action} $\beta$ of
$\gg$ on $M$, recorded as $(\gg, M,\beta)$, is a linear map
$X\mapsto X^\beta$ from $\gg$ to $\vv^r (M)$ that
commutes with Lie brackets and whose evaluation map is $C^r$.
Unless otherwise indicated it is tacitly assumed that $r=\infty$ or
$\omega$.  The action is {\em complete} provided each vector field
$X^\beta$ is complete, i.e., all its integral curves extend over
$\RR$. 

An
{\em $n$-action} (of a Lie group or Lie algebra)  is an action on
an $n$-dimensional manifold. 

A $C^s$ {\em local action} $\lam$ of $G$ on $V$, $(0\le s\le \omega)$ is a
homomorphism $g\mapsto g^\lambda$ from $G$ to the groupoid of $C^s$
diffeomorphisms 
between open subsets.
having the following properties:  The evaluation map $(g, x)\mapsto
g^\lambda (x)$ defines a $C^s$ map $\Omega \to V $, where $\Omega$ is
an open neighborhood of $\{e\}\times V$. 
Suppose  $s>1$.  Corresponding to $\lambda$ is a $C^{s-1}$
action of $\gg$ on $V$ denoted by $\msf d\lam$.  
Conversely, every  $C^s$ action
of $\gg$ comes from a 
$C^s$ local actions of $G$.  When 
$G$ is simply connected and the  Lie
algebra action
 $(\gg, M,\beta)$ is a complete, then there exists an action $(G, M, \alpha)$ such that
$\beta=\msf d\alpha$.   For results on the
smoothness of these actions see Hart \cite{Hart81, Hart83}, 
Stowe \cite{Stowe83}.

The {\em orbit} of $p\in M$ under $(G,M,\alpha)$ is $\{g^\alpha (p)\co
g\in G\}$, and the orbit of $p$ under a Lie algebra action
$(\gg,M,\beta)$ is the union over $X\in \gg$ of the integral curves of
$p$ for $X^\beta$.  An action is {\em transitive} if
it has only one orbit.

The  {\em fixed point set} of $(G, M, \alpha)$ is the set 
\[\Fix {\alpha}:=
\{ x\in M\co g^\alpha (x)=x, \ (g\in G)\},
\]
denoted also by $\Fix {G^\alpha}$. 
For Lie algebra actions $( \gg,M,\alpha)$ the fixed point set is 
\[ 
 \Fix{\beta}:= \Fix {\gg^\beta} :=\{p\in M\co X^\beta_p=0, \ (X\in \gg)\}
\]
Thus $p\in\Fix \beta$ iff $p$ is a fixed point for the
local flows on $M$ defined by the vector fields $X^\beta$ for all
$X\in \gg$. 

The {\em support} of an action $\gamma$ on $M$ is the closure of
$M\,\verb=\=\,\Fix \gamma$.   

An action $\alpha$ is {\em effective} if  $\ker (\alpha)$ is
trivial, and {\em nondegenerate} if the fixed point set of every
nontrivial element has empty interior.  A degenerate action $\alpha$
of $\gg$ is trivial if $\alpha$ is analytic or $\gg$ is simple.

A group action is {\em almost effective} if its
kernel is discrete.

\section{Constructions of  actions} 
 \label{sec:construction}

\subsection{Analytic actions of $\R n$}    \label{sec:abelian}
It is true, but not easy to prove, that every real analytic manifold
admits a nontrivial analytic vector field.\footnote{This is false for
complex manifolds, e.g., Riemann surfaces of genus $>1$.} 
In
fact the following holds:
\begin{theorem}         \label{th:Rn}
The vector group $\R n$ has an effective analytic action on
every real analytic manifold $M$ of dimension $\ge 2$.
\end{theorem}
The proof relies on the theory of approximation of
smooth functions by analytic functions (Grauert \cite{Grau58}).
\begin{lemma}           \label{th:rnLem}
Let $f\co M\to \RR$ be a nonconstant analytic function that is
constant on each boundary component.  Then there exists a nontrivial
analytic vector field $X$ on $M$ such that $df( X_p) =0$ for all $p\in
M$, and $X$ generates an analytic flow that preserves each level set
of $f$.
\end{lemma}
\noindent{\em Proof.}\footnote{Joint work with Professor Joel Robbin.}
First consider the case
that $M$ is an open set $W$ in a half-space $
H^d=[0,\infty)\times \R {d-1} \subset \R d$, so that $\partial W=
W\cap (\{0\}\times \R {d-1})$.
Fix a nonconstant analytic function
$f\co W\to \RR$ that is constant on  each boundary component.
The analytic vector field $Y=\pde f {x_2} \frac{\partial}{\partial
x_1} - \pde f {x_1} \frac{\partial}{\partial x_2}$ is nontrivial,
annihilates $df$, and is tangent to $\p W$.
Endow $W$ with  a complete Riemannian metric and  denote the norm 
of  $\xi\in TW$ by $\vert \xi\vert$. 
There is a nonconstant analytic function
$u\co W\to \RR$ such that 
$\sup_{q\in M}|u(q)Y_q| <\infty$.  The vector field $X=uY$ is
complete, and satisfies the Lemma. 

Now let $M$ be arbitrary.  By Whitney's embedding
theorem and the tubular neighborhood theorem (see Hirsch
\cite{Hirsch76}) we take $M$ to be an analytic submanifold of some
halfspace ${H}^d\subset \R d$ such that $ M\cap\partial {H}^d=\partial M$,
with an open neighborhood $W\subset {H}^d$ of 
$M$ having an analytic retraction $\pi\co W\to M$ taking $\partial M$
into $\partial {H}^d$.
By the first part of the proof, there is a complete, nontrivial
analytic vector field $U$ on $W$ that annihilates the function
$f\circ\pi \co W\to\RR$.  For $p\in M$ set $U_p=X_p + Z_p$ with $X_p,
Z_p\in T_pM$ and $\msf d\pi_p Z_p=0$.
The maps $p\mapsto U_p$ are analytic vector fields on $M$.  We
have
\[\begin{split}
0 &=d(f\circ\pi)_p U_p=df_p\circ \msf d\pi_p X_p=
df_p \circ d(\pi|M)_p X_p,\\
  &=df_p X_p
\end{split}
\]
because $\pi|M$ is the identity map. 
\qed

\medskip\noindent {\em Proof of Theorem \ref{th:Rn}. }  Choose an
analytic vector field $X$ on $M$ as in the Lemma.  Fix analytic
functions $u_j\co M\to\RR,\;\:j=1,\dots,n$ that are linearly
independent over $\RR$ such that $|u_j X|$ is bounded. The vector
fields $L_j$ on $M$ defined by
  $L_j(p)= u_j(f (p)) X_p$  
are complete and therefore generate flows $\phi_j$.  In each level set
$V$ of $f$, $L_j|V$ is a constant scalar multiple of $X|V$.  Therefore
$\phi_j$ preserves $V$, and $[L_j, L_k]=0$.  This shows that the
$\phi_j$ generate an analytic action $\Phi$ of the group $\R n$.

To show that $\Phi$ is effective, assume $a_j\in \RR$ are such that
$\sum_j a_j L_j$ vanishes identically, which means 
$(\sum_j a_j u_j(f (p)))
X_p$ vanishes identically.  So therefore does $\sum_j a_j u_j(f (p))$,
by analyticity, because $X_p\ne 0$ in a dense open set.  It follows that
the $a_j$ are  zero because the  $u_j$ are linearly
independent.  \qed

\subsection{Lie algebra actions on noncompact manifolds } 
  \label{sec:open}
A manifold is {\em open} if it is connected, noncompact and without boundary.
On many  open manifolds  it is
comparatively easy to produce  Lie algebra actions that are effective
and analytic:
\begin{theorem}		\label{th:noncompact}
  An open manifold $M^n$ admits an effective Lie algebra action $( \gg,
  M^n,\beta)$ if there is an effective action $(\gg, W^n,\alpha)$ such
  that one of following conditions is satisfied:
\begin{description}

\item[(a)] $M^n$ is parallelizable (which holds if $n=2$ and $M^n$ is
  orientable)

\item[(b)] $n=2$  and $W^2$ is nonorientable. 

\end{description}
In each case $\beta$ can be chosen to be nondegenerate, analytic,
transitive or fixed-point free provided $\alpha$ has the same
property.

\end{theorem}
\begin{proof} We will define $\beta$ as the pullback of $\alpha$ by an
analytic immersion $M^n\to W^n$.  The fundamental theorem of immersion
theory (Hirsch \cite{Hirsch59, Hirsch61}, Poenaru \cite{Poenaru62},
Adachi \cite{Adachi84})) says that such an immersion exists provided
$M^n$ is an open manifold, and the tangent bundle $TM^n$ is isomorphic
to the pullback of $TW^n$ by map $f\co M^n \to W^n$.  In case (a) take
$f$ to be any constant map.  For case (b) we first show that $M^2$
immerses in the M\"obius band $B^2$.  To see this, note that every
open surface has the homotopy type of a $1$-dimensional simplicial
complex, whence the classification of vector bundles implies $TM^2$ is
the pullback of $TP^2$ by a map $f$ from $M^2$ to the projective plane
$P^2$.  As $M^2$ is an open surface, it can be deformed into an
arbitrary neighborhood of its $1$-skeleton, hence $f$ can be chosen to
miss a point of $P^2$ and thus have its image in a M\"obius band. This
shows that $M^2$ immerses in $B^2$, and the conclusion follows because
$B^2$ immerses in every nonorientable surface.
\end{proof}

 The real form of a
Lie algebra $\gg$ of matrices over $\CC$ or the quaternions $\HH$ is denoted by
$\gg_\RR$.  From  the natural projective actions of matrix groups we obtain:
\begin{corollary}		\label{th:opencor}
The following effective kinds of analytic actions exist:
\begin{description}
  
\item[(a)]  $\ss\ll (3,\RR)$ and $\ss\ll (2,\CC)_\RR$ 
  on all open surfaces, and on all open 
  parallelizable $k$-manifolds for $k\ge 3$,

\item[(b)]$\ss\ll (n, \RR)$  on all
open  parallelizable $k$-manifolds,  $k\ge n-1$,

\item[(c)] $\ss\ll (n, \CC)_\RR$  on
  all open parallelizable $k$-manifolds, $k\ge 2n-2$,

\item[(d)] $\ss\ll (n, \mathbb H )_\RR$  on
  all open parallelizable $k$-manifolds, $k\ge 4n-4$.

\end{description}
\end{corollary}

\section{Lie contractible groups}  \label{sec:AC}
Let $\mcal G$ denote either a Lie group $G$ or its Lie
algebra $\gg$.  A {\em deformation} of $\mcal G$ is a $1$-parameter
family $\theta=\{\theta_t\}_{t\in\RR}$ of endomorphisms $\theta_t\co
\mcal G\to \mcal G$ having the following properties:

\begin{description}
\item[(D1)]$\theta_t$ is the identity automorphism if $t\le 0$,

\item[(D2)] $\theta_t = \theta_1$ if $t\ge 1$. 

\item[(D3)] the map $\RR\times \mcal G \to \mcal G$, \, $(t,g)\mapsto
  \theta_t (g)$ is $C^\infty$.
\end{description}

If  $\theta$ is a deformation of $G$, the family 
 $\theta'_t (e)\}_{t\in\RR}$ of 
derivatives at
the unit element $e\in G$  constitute a  deformation $\theta'$
 of $\gg$.  When $G$ is simply connected every
deformation of $\gg$ comes in this way from a unique deformation of
$G$.

An {\em Lie contraction} of $\mcal G$ is a deformation $\theta$ such
that $\theta_1$ is the trivial endomorphism, in which case $\mcal G$
is called {\em Lie contractible}.  It can be shown that this implies
$\mcal G$ is solvable, and $G$ is contractible as a topological space.
It is easy to see that the direct product of finitely many Lie
contractible groups is Lie contractible.
 
We prove below that $\ss\ttt (n,\FF)$ and its commutator ideal, which
is nilpotent, are Lie contractible.  But K. DeKimpe \cite{DeKimpe06}
pointed out that some nilpotent Lie algebras have unipotent derivation
algebras,  ruling out Lie
contractibility.  Goodman \cite{Goodman76} cites an example due to
M\"{u}ller-R\"{o}mer \cite{MullerRomer76} of such an algebra, namely,
the $7$-dimensional Lie algebra with a basis such that
\[\begin{split}
[X_1, X_k] &= \, X_{k+1}, \quad (k=2,\dots,6), \\
[X_2,X_3]  &= \,X_6,\quad [X_2,X_4]  = \,X_7,\\
[X_3, X_4] &= \, X_7,\quad [X_2, X_5] = -X_7. 
\end{split}
\]
See also Ancochea \& Campoamor 
\cite{AncocheaCampoamor01}, Dixmier \& Lister
\cite{DixmierLister57}, 
Dyer \cite{Dyer70}.

Let $\mcal H \subset \mcal G$ be a subalgebra or subgroup.
A deformation $\theta$ of $\mcal G$ is a  {\em retraction} of  $\mcal
G$ into $\mcal H$ provided
\[ 
\theta_1 (\mcal G)\subset \mcal H,  \quad 
 \theta_t(\mcal H)\subset \mcal H \ \mbox{for all $t\in\RR$}.
\]
When such a $\theta$ is given and $\psi$ is a deformation of $\mcal
H$, there is deformation $\psi\#\theta$ of
$\mcal G$, and the  {\em concatenation} of $\psi$ and $\theta$, defined by
\[
  (\psi\#\theta)_t = \begin{cases}
    \theta_{2t}              & \mbox{if $-\infty <t\le 1/2$},\\
    \psi_{2t-1}\circ \theta_1 & \mbox{if $ 1/2 \le t <\infty$}
\end{cases}
\]
is also another retraction of $\mcal G$ int $\mcal H$.
Note that $(\psi\#\theta)_1= \psi_1\circ \theta_1$.  If $\theta$ is a
Lie contraction of $\mcal H$ then $\psi\#\theta$ is a Lie
contraction of $\mcal G$.

\begin{theorem}         \label{th:liecont}
The groups $\w{ST}_\circ (N,\FF)$,   $ST' (n,\FF)$ and their Lie algebras are
 Lie contractible. 
\end{theorem}
\begin{proof}
It suffices to prove the Lie algebras are Lie contractible. 
Give $\ttt (n,\FF)$ the basis $\{T^{(ij)}\}_{1\le i\le j\le m}$ where
the unique nonzero entry of the $n \times n$ matrix $T^{(ij)}$ is $1$
in row $i$, column $j$.  The matrices $T^{(ij)}$ with $i <j$ form a
basis for the commutator ideal $\ss\ttt (n,\FF)'$, and
$\{T^{(11)},\dots,T^{(n-1,n-1)}\}$ is a basis for the subalgebra $\dd
(n,\RR)\subset\ss\tt (n,\RR)$ of diagonal matrices.

For $1\le i < j \le n$ and $t\in\RR$, the functions 

\begin{equation}                \label{eq:cij}
c_{ij}\co\RR\to[0,1], \qquad
c_{ij}(t):=
\begin{cases} 
        
               1      &\text{if}  \ t \le 0,\\
\exp\frac{(j-i)t}{1-t}  &\text{if} \ 0<t<1,\\
                0     &\text{if}  \ t \ge 1
\end{cases}
\end{equation}
are $C^\infty$,  analytic in the open interval $]0,1[$, and flat at
  all $t\notin ]0,1[$. They satisfy

\begin{equation}                \label{eq:cijk}
 c_{ij}(t)\cdot c_{jk}(t)=c_{ik}(t) \qquad (i\le j \le k).
\end{equation}

Consider the $1$-parameter family of  linear maps
\[ \theta_t\co\ss\ttt (n,\FF)\to \ss\ttt (n,\FF), \qquad (t\in\RR),\]
defined on the basis elements by
\begin{equation}                \label{eq:phitT}
\theta_t (T^{(ij)})=
\begin{cases} T^{(ij)}          & \mbox{if $t\le 0$ or $i=j$,}\\
              c^{(ij)}(t)\cdot T^{(ij)}& \mbox{if $0< t <1$ and $i <j$,}\\
              0        & \mbox{if $t\ge 1$ and $i< j$.}
\end{cases} 
\end{equation}
The first equation in (\ref{eq:phitT}) implies $\theta$ maps $\ss\ttt
(n,\FF)'$ into itself and reduces to the identity deformation of $ \dd
(n,\RR)$.  

Equation (\ref{eq:phitT}) defines a  retraction $\theta$ of $\ss\ttt
(n,\FF)$ into $\dd (n,\FF)$, thanks to (\ref{eq:cijk}), and $\theta$
restricts to a Lie contraction $\theta^1$ of $\ss\ttt (n,\FF)'$.  To
obtain a Lie contraction of $\ss\ttt (n,\FF)$ it suffices to form a
concatenation $\psi\#\theta$ where $\psi$ is an algebraic
contraction of $\dd (n,\FF)$.  For example set
\[
   \psi_t(T^{(ii)})=c(t)\cdot T^{(ii)},\qquad (i=1,\dots,n) 
\]
where  $c(t):=c_{21}(t)$  from Equation (\ref{eq:cij})). 
\end{proof}

\subsection{Deformations of actions}   \label{sec:defactions}
Let $\alpha_0, \alpha_1$ be actions of $G$ on $M$. A {\em
  deformation} of $\alpha_0$ to $\alpha_1$ is a $1$-parameter family
of actions $\beta=\{(G, M,\beta_t)\}_{t\in\RR}$ such that
\begin{itemize}
\item $\beta_t = \alpha_0$, \ $(t\le 0)$

\item $\beta_t=\alpha_1$, \ $(t\ge 1)$

\item the map
$\RR\times G\times  M \to M, \, (t,g,x)\mapsto g^{\beta_t} (x)$, 
is $C^\infty$. 
\end{itemize}
A  Lie contraction  $\theta$ of $G$ determines  the deformation
$\beta$  of 
$\alpha_0$ to the trivial action, defined by
\[
  \beta_t =\alpha_0\circ \theta_t, 
\]
indicated by 

\begin{theorem}         \label{th:ST1}
Let $G$ be a  Lie contractible group having an almost effective
smooth action on $\S{n-1}$.  Then on every  topological $n$-manifold $M^n$ there
is an  effective action of $G$ which is the identity outside a
coordinate ball, and which is smooth if $M^n$ is smooth.  
\end{theorem}
\begin{proof} Let  $\theta$ be a Lie contraction of an effective
  smooth action $(G, \S{n-1}, \alpha)$.   An action\\
$(G, \S{n-1}\times \RR, \beta)$ is defined by
\[
  g^\beta\co (x, t) = (\theta_t(g)^\alpha (x), t).
\]
$\beta$ is smooth and effective, and
\[
  g^\beta\co (x, t) =
\begin{cases}
             (g^\alpha (x), t) & \text{if} t\le 0,\\
             (x,t)            & \text{if} t \ge 1.
\end{cases}
\]  
Transfer $\beta$ to an action $(G, \R n\verb=\=\{0\}, \gamma_0)$ by 
by the diffeomorphism 
\[
 \R n\verb=\=\{0\} \to \S{n-1}\times \RR,\qquad  (x, t)\mapsto e^{-t}x.
\] 
This action extends to a smooth effective action $(G, \R n, \gamma)$
which is the identity outside the unit ball.  It can therefore be
transferred to the desired action on $M^n$. 
\end{proof}

\begin{corollary}         \label{th:ST3}
For all $n\ge 2, \,k \ge 1$ there are effective smooth actions of
 $ST_\circ (n, \RR)^k$ on every smooth $n$-manifold
 and of $\w{ST}
(n,\CC)$ on every smooth $2n$-manifold.
\end{corollary}
\begin{proof} $ST_\circ (n,\RR)$ and $\w{ST} (n,\CC)$ are Lie
  contractible (Theorem \ref{th:liecont})
 and have effective smooth
  actions on $\S {n-1}$ and $\S{2n-1}$ respectively.
  By Theorem
  \ref{th:ST1} there are $k$ coordinate balls with disjoint closures
  in $M^n$ (respectively, $M^{2n}$) that support effective smooth
  actions of $ST_\circ (n, \RR)$ (respectively, x$\w{ST} (n,\CC))$.  The
  desired actions are obtained by letting $j$'th factor of the direct
  product act smoothly and effectively in  the $j$'th coordinate,
  ball and trivially outside it.
 \end{proof}

\section{The Epstein-Thurston obstruction to effective solvable
  actions}   \label{sec:ET} 
In this section $G$ can be either real or complex.  In the complex
case an $n$-action means a holomorphic action on a complex
$n$-dimensional manifold. 

D.B.A. Epstein and W.P. Thurston \cite[Theorem 1.1] {ET79} discovered
a fundamental necessary condition for effective local actions of
solvable Lie groups:\footnote{
The authors point out that their proof, stated for the real field, is
valid in any category having a ``good'' dimension theory.  It probably
works  for algebraic actions over  arbitrary fields.} 
\begin{theorem} \label{th:ET}
Assume   $G$ is solvable and has an effective local $n$-action.  Then $n\ge
\ell (G) -1$, and $n\ge \ell (G)$ if $G$ is nilpotent.
\end{theorem}
The same conclusions hold for solvable Lie algebras actions.  It turns
out that in the borderline dimensions there are further restrictions
on the structure of $G$ and its orbits:

\begin{theorem}         \label{th:ETcor}
Assume $G$ is solvable with derived length $l$.  Let $(G, M^n,
\alpha)$ be a nondegenerate local  action,  with $n=l$ if
$G$ is nilpotent and $n=l-1$ otherwise.

\begin{description} 
\item[(i)]  The union $W$ of the open orbits is dense.

\item[(ii)] Suppose $G^{(l-1)}$ lies in the center $C$ of $G$.
  Then $\dim (G^{(l-1)}) = \dim
  (C)=1$, and  $G^{(l-1)} =  C$. 

\end{description}
\end{theorem}
\begin{proof}
{\em (i) } $G^{(l-1)}$ acts trivially in each orbit of dimension $<n$
by the Epstein-Thurston theorem.  As the action is effective, there is
an orbit $U$ in which $G^{(l-1)}$ acts nontrivially.  The
Epstein-Thurston theorem implies $U$ is $n$-dimensional.  This shows
that $W$ is nonempty.  Each orbit $M\verb=\=W$ has dimension $<n$,
therefore $G^{(l-1)}$ acts trivially in $M\verb=\=W$.  Nondegeneracy
implies $M\verb=\=W$ contains no open set, hence $W$ is dense.

{\em (ii) } Fix a $1$-dimensional subspace $Z\subset C$.  In view of
(i) we assume $\alpha$ is transitive and the orbits of $Z^\alpha$ are
the fibres of a trivial fibration of $\pi\co M^n\to V^{n-1}$.  Let
$(G/Z, V^{n-1}, \beta)$ be the action related equivariantly to $\alpha$
by $\pi$.  The Epstein-Thurston theorem implies $\beta|G^{(l-1)}$ is
trivial, hence $\alpha$-orbits of $G^{(l-1)}\times Z$ are the
$1$-dimensional orbits of $Z^\alpha$.  Let $K\subset G^{(l-1)}\times
Z$ be the stabilizer of some point of $M^n$ under the action of
$\alpha|(G^{(l-1)}\times Z)$.  Centrality of $G^{(l-1)}\times Z$ and
transitivity of $\alpha$ imply $K^\alpha$ stabilizes every point of
$M^n$, and is therefore trivial because $\alpha$ is effective.
Consequently $\dim(G^{(l-1)}\times Z)= 1$, which implies (ii) because
$G^{(l-1)}$ is nontrivial by the Epstein-Thurston Theorem.
\end{proof}

Examination of the proof yields:
\begin{corollary}               \label{th:}
Assume  $G, l$ and $n$  satisfy the hypothesis of Theorem
\ref{th:ETcor}, the center $C$ of $G$ contains $G^{(l-1)}$, and
$\dim C >1$.  Then the kernel of any analytic $n$-action of $G$
contains a $1$-dimensional central subgroup.  \qed
\end{corollary}

\begin{example}         \label{th:snp1}
 For all $n\ge 2$ the group $N(n,\FF):=ST (n, \FF)'\times \FF$ has
 effective smooth actions on every $n$-manifold (Theorem
 \ref{th:ST3}).  The actions constructed in the proof are highly
 degenerate, and in fact:

\begin{itemize}
\item {\em Every $n$-action of $N(n,\FF)$ is degenerate.} 
\end{itemize}
This follows from Theorem \ref{th:ETcor}(ii): \ $N(n,\FF)$ is
nilpotent and with   derived length $n$, and its center is
$2$-dimensional (over $\FF$) and contains the $1$-dimensional subgroup
$N_n^{n-1}$.
\end{example}

\section{Semisimple actions}   \label{sec:semisimple} 
Let $(G, M, \alpha)$ be an analytic action of a semisimple group.  
If the linearization $\msf d\alpha_p$ at $p\in \Fix G$   is trivial
then $\alpha$  is trivial, because in a neighborhood of $p$,  $\alpha$ 
is analytically equivalent to  $\msf d\alpha_p$  (A. Ku\v{s}nirenko
\cite{Ku67}, V. Guillemin \& S. Sternberg \cite{GS68}, R. Hermann
\cite{Hermann65}).  

Cairns \& Ghys \cite{CairnsGhys97} constructed effective $C^\infty$
actions of $SL(2,\RR)$ on $\R 3$ and $SL(3,\RR)$ on $\R 8$ with fixed
points at the origin, at which they are 
not topologically locally conjugate to 
analytic actions.  Nevertheless the same conclusion holds for $C^1$
actions by a  striking result, W. P. Thurston's  ``Generalized Reeb
Stability Theorem,'' \cite{Thurston74}:

\begin{theorem}[{\sc Thurston}]         \label{th:thurston}
If $\alpha$ is a  nontrivial local  $C^1$ action of a semisimple
Lie group,  at every  fixed point $p$ the linearized action
$\msf d\alpha_p$ is 
nontrivial. 
\end{theorem}
\begin{proof}  While Thurston  states his theorem for global Lie group
  actions,  the proof is entirely local. 
\end{proof}

Other results on semisimple actions are given in the papers cited
above, and in T. Asoh \cite{Asoh87}, C. Schneider \cite {Schneider74},
D. Stowe \cite{Stowe83}, Uchida \cite{Uchida79, Uchida97}, Uchida \&
Mukoyama \cite{UchidaMukoyama02}.

Example  \ref{th:snp1} showed that all $n$-actions of $ST (n,
\RR)'\times \RR$ are degenerate.  This phenomenon cannot occur for
effective $C^1$ actions by semisimple groups: 

\begin{theorem}             \label{th:thurston1}
Let $G$ be a semisimple Lie group and $(G, M^n, \alpha)$ an effective
$C^1$ local action.  Then $\alpha$ is nondegenerate, as is the induced
action in $\p M^n$.
\end{theorem}
 \begin{proof} It suffices to consider a $C^1$ local group action
   $(G, M^n,\alpha)$.  For every invariant set $L\subset M$ let $(G, L,
   \alpha_L)$ be the action induced by $\alpha$.  Fix a nonempty open set
   $U\subset M^n$ and let $K\subset G$ denote the kernel of
   $\alpha_U$.  Every point $p\in U$ is a fixed point of $\alpha|K$ at
   which the linearized action $\msf d\alpha_p|\kk$ is trivial.  As
   $K$ is normal in $G$ and therefore semisimple, Thurston's theorem
   applied to $(K, M^n, \alpha|K)$ shows that $K\subset \ker
   (\alpha)$.  Therefore $K$ is the trivial subgroup because $\alpha$
   is effective, proving that $\alpha$ is nondegenerate.

Assume {\em per contra } that $\alpha_{\p M}$ is degenerate.  The
preceding paragraph shows that there is a nontrivial proper normal
subgroup $H\subset G$ such that $H^\alpha$  acts trivially on $\p M$.  Let
$(H, M^n.\gam)$ be the action induced by $\alpha$.  At every $p\in\p
M^n$ there is an analytic 
coordinate chart centered at $p$ taking a neighborhood of $p$ onto an
open subset of the origin in the closed half-space of $\R n$ defined
by $x_n \ge 0$.  In these coordinates $\msf d\gamma_p$ represents $H$
in the abelian subgroup comprising the  matrices $A\in GL (\R n)$
having the block form 
$ 
 \left[\begin{smallmatrix} 
   I_{n-1} & b\\ 0 & 1
  \end{smallmatrix}\right]$.
 Semisimplicity of $H$ implies $\msf d\gam_p$ is trivial.  Therefore
 $\gamma$ is trivial by Thurston's theorem, contradicting
 effectiveness of $\alpha$.
\end{proof}

Here is another application of Thurston's result: 
\begin{theorem}               \label{th:SL}
 If $\p M^n\ne\varnothing$, every $C^1$ local action $(SL_\circ (n+1,\RR),
 M^n, \alpha)$ is trivial.
\end{theorem}
\begin{proof} The Epstein-Thurston Theorem \ref{th:ET} implies  
the subgroup $ST_\circ (n+1,\RR)$ does not have  effective local actions on
 $(n-1)$-manifolds.  Therefore  $\alpha_{\p M^n}$ is degenerate,
so $\alpha$ is trivial by Theorem \ref{th:thurston1}.
\end{proof}

\begin{example}         \label{th:exSL2}
Theorem \ref{th:SL} shows that $\w{SL}_\circ (2,\RR)$ does not have 
effective $C^1$ local actions on the compact interval $[0,1]$.  On the
other hand:
\begin{itemize}
\item{\em  $\w{SL}_\circ (2,\RR)$
  has nondegenerate  continuous actions on
$[0,1]$. }
\end{itemize}
To construct such an action,  identify the open unit interval $]0,1[$ with a
universal covering space of $\S 1$, lift  the natural action
$(SL_\circ (2,\RR), \S 1, \alpha)$ to an action $(\w{SL}_\circ
(2, \RR), ]0,1[, \beta)$,  and extend  $\beta$ to an action  
$(\w{SL}_\circ (2, \RR), ]0,1[, \delta)$.

By putting $\delta$ on each radius of the compact $n$-disk $\D n$,
for every $n$ we get a  nondegenerate  action of  $\w{SL}_\circ (2,\RR)$
on $\D n$ that is trivial on $\p\D n$.  This leads to:
\begin{itemize}
\item{\em $\w{SL}_\circ (2,\RR)$ acts nondegenerately on all
  $CW$-complexes.}\footnote{Professor Palais said  this
  construction is a ``cheat''.} 
\end{itemize}
\end{example}

\bibliographystyle{amsplain}

\bigskip

\flushleft
-----------------------------------------------------------\\

\end{document}